\documentclass[final]{amsart}
\usepackage[draft]{showkeys} 
\usepackage{mathtools}
\usepackage{amsfonts}
\usepackage{amssymb}
\usepackage[latin2]{inputenc}
\usepackage{enumerate}
\usepackage{amsmath}
\usepackage{amssymb}
\usepackage{framed,xcolor}
\usepackage{lipsum,hyperref}

\colorlet{shadecolor}{yellow!20}

\numberwithin{equation}{section}
\newtheorem{theorem}{Theorem}[section]
\newtheorem{lemma}[theorem]{Lemma}

\newtheorem{prop}[theorem]{Proposition}

\newtheorem{fff}[theorem]{Fact}
\newtheorem{question}[theorem]{Question}
\newtheorem{problem}[theorem]{Problem}
\newtheorem{cor}[theorem]{Corollary}

\newcommand{\thistheoremname}{}
\newtheorem*{genericthm*}{\thistheoremname}
\newenvironment{namedthm*}[1]
  {\renewcommand{\thistheoremname}{#1}%
   \begin{genericthm*}}
  {\end{genericthm*}}

\theoremstyle{definition}       {         
\newtheorem{remark}[theorem]{Remark}
\newtheorem{remarks}[theorem]{Remarks}

\newtheorem{defi}[theorem]{Definition}
\newtheorem{notation}[theorem]{Notation}
}

\newcommand{\N}{\mathbb{N}}
\newcommand{\Z}{\mathbb{Z}}

\newcommand{\R}{\mathbb{R}}

\newcommand{\HH}{\mathcal{H}}

\newcommand{\B}{\overline{B}}

\newcommand{\eps}{\varepsilon}

\def\beq{\begin{equation}}
\def\eeq{\end{equation}}
\newcommand{\semmi}[1]{{}}

\newcommand{\bp}{\begin{proof}}
\newcommand{\ep}{\end{proof}}

\DeclareMathOperator{\ubdim}{\overline{dim}_B}
\DeclareMathOperator{\hdim}{\dim_H}
\DeclareMathOperator{\pdim}{\dim_P}

\DeclareMathOperator{\diam}{diam}

\DeclareMathOperator{\dist}{dist}

\DeclareMathOperator{\Sym}{Sym}

\DeclarePairedDelimiter{\ceil}{\lceil}{\rceil}
\DeclarePairedDelimiter{\floor}{\lfloor}{\rfloor}

\newcommand{\su}{\subset}

\newcommand{\NN}{\mathbb{N}}
\newcommand{\QQ}{\mathbb{Q}}
\newcommand{\RR}{\mathbb{R}}

\newcommand{\iH}{\mathcal{H}}

\newcommand{\iF}{\mathcal{F}}

\newcommand{\sm}{\setminus}

\newcommand{\al}{\alpha}
\newcommand{\be}{\beta}

\title{Lipschitz images and dimensions}
\author{Rich\'ard Balka} 
\address{HUN-REN Alfr\'ed R\'enyi Institute of Mathematics, Re\'altanoda u.~13--15, H-1053 Budapest, Hungary, AND Institute of Mathematics and Informatics, Eszterh\'azy K\'aroly Catholic University, Le\'anyka u.~4, H-3300 Eger, Hungary}
\email{balkaricsi@gmail.com}

\thanks{The authors were supported by the National Research, Development and Innovation Office -- NKFIH, grants no.~124749 and 146922. 
The first author was also supported by the J\'anos Bolyai Research Scholarship of the Hungarian Academy of Sciences.
The second author was also supported by the
National Research, Development and Innovation Office -- NKFIH, grant no.~129335.}
\author{Tam\'as Keleti}
\address{Institute of Mathematics, ELTE E\"otv\"os Lor\'and University, P\'azm\'any P\'eter s\'et\'any 1/c, H-1117 Budapest, Hungary}
\email{tamas.keleti@gmail.com}



\begin{document}

\begin{abstract} We consider the question which compact metric spaces can be obtained as a Lipschitz image of
the middle third Cantor set, or more generally, as a Lipschitz image of a subset of a given compact metric space. 

In the general case we prove that if $A$ and $B$ are compact metric spaces and the Hausdorff dimension of $A$ is bigger than the upper box dimension of $B$, then there exist a compact set $A'\subset A$ and a Lipschitz onto map $f\colon A'\to B$. 

As a corollary we prove that any `natural' dimension in $\R^n$
must be between the Hausdorff and upper box dimensions.

We show that if $A$ and $B$ are self-similar sets with the strong separation condition with equal Hausdorff dimension and $A$ is homogeneous, then $A$ can be mapped onto $B$ by a Lipschitz map if and only if $A$ and $B$ are bilipschitz equivalent. 

For given $\alpha>0$ we also give a characterization of those compact metric spaces that can be obtained as an $\alpha$-H\"older image of a compact subset of $\R$.
The quantity we introduce for this turns out to be closely related to the upper box dimension.

\end{abstract}
\maketitle

\section{Introduction}

The characterization of compact sets in $\RR^2$ which can be covered by a Lipschitz image of $[0,1]$ was given in the celebrated paper of P.~W.~Jones \cite{J}, which was generalized to $\RR^d$ by K.~Okikiolu \cite{O}. This is called the analyst's traveling salesman theorem. 
Later R.~Schul \cite{S} and I.~Hahlomaa \cite{H} considered the question in Hilbert spaces and general metric spaces, respectively. Motivated by these results, we formulated the following question.  

\begin{question}\label{q:intro}
(a) What compact metric spaces can be obtained as a Lipschitz image of a given compact metric space $X$ or at least as a Lipschitz image of a subset of $X$? 

(b) For example, what 
compact metric spaces can be obtained as a Lipschitz image of
the middle third Cantor set?
\end{question}

\subsection{Lipschitz images of general compact metric spaces}

M.~Laczkovich~\cite{Lacz} posed the following problem.

\begin{problem}[Laczkovich, 1991] \label{p:L} Let $A\subset \mathbb{R}^{d}$ be a measurable set with positive $d$-dimensional Lebesgue measure. Is there a Lipschitz onto map $f\colon A\to [0,1]^{d}$?
\end{problem}

For $d=1$ the positive answer is only an easy exercise.
For $d=2$ the affirmative answer was given first by D.~Preiss~\cite{P}.
By modifying partly the argument of Preiss, J.~Matou\v{s}ek~\cite{M} provided a simpler proof based on a well-known combinatorial lemma due to Erd\H{o}s and Szekeres. Meanwhile, P.~Jones noticed that the $d=2$ case easily follows from an earlier result of N.~X.~Uy \cite{Uy}, see the argument of Jones in \cite{ACP} after Theorem~2.6.
Problem~\ref{p:L} is still open for dimensions $d\geq 3$. 

The following result of A.~G.~Vitu\v{s}kin, L.~D.~Ivanov, and M.~S.~Melnikov~\cite{VIM} shows (see also \cite{K} for a less concise proof) that Laczkovich's problem does not remain true if we only assume that $E$ has positive $d$-dimensional Hausdorff measure 
in a larger dimensional space. 

\begin{theorem}[Vitu\v{s}kin--Ivanov--Melnikov, 1963] There exists a compact set $K$ on the plane such that 
$K$ has positive $1$-dimensional Hausdorff measure
but there exists no Lipschitz onto map $f\colon K\to [0,1]$. 
\end{theorem}

Let $\hdim$ and $\ubdim$ denote the Hausdorff and upper box dimensions, respectively. 
Keleti, M\'ath\'e, and Zindulka~\cite{KMZ} proved the following positive result
under the stronger assumption $\hdim A>d$.

\begin{theorem}[Keleti--M\'ath\'e--Zindulka, 2014] \label{t:KMZ}
Let $A$ be a
compact metric space 
with $\hdim A>d$. Then there exists a Lipschitz onto map $f\colon A\to [0,1]^d$. 
\end{theorem}

Our first main result is the following
generalization of Theorem~\ref{t:KMZ}. 
Here we state it only for Lipschitz maps 
but in fact we prove a more general statement for H\"older maps. 

\begin{theorem}
\label{t:dimimplieslip} 
Let $A$ 
and $B$ be compact metric spaces such that $\hdim A > \ubdim B$.
Then there exists a compact set $A'\su A$ and a Lipschitz onto map $f\colon A'\to B$.
\end{theorem}

In order to get a Lipschitz onto map defined on the whole $A$ we need to extend the Lipschitz
function obtained in Theorem~\ref{t:dimimplieslip} to $A$.
To this end, we prove the 
following lemma, which shows that this is 
always possible if $A$ is an ultrametric space
(see the definition in Section~\ref{s:prel}).

\begin{lemma}\label{l:extension}
Let $X$ be a compact ultrametric space and $Y$ be any complete metric space.
Then for any $A\su X$ any Lipschitz function $f\colon A\to Y$ 
can be extended to $X$ as a Lipschitz function with the
same Lipschitz constant.
\end{lemma}

Since it is well known and easy to show (see Lemma~\ref{sscbilipultra}) that
any self-similar set with the strong separation condition (SSC) is bilipschitz equivalent
to an ultrametric space (see the definitions in Section~\ref{s:prel}),
Theorem~\ref{t:dimimplieslip} and Lemma~\ref{l:extension} immediately give
the following.

\begin{cor}\label{c:dimwithextension}
Let $A$ be a complete ultrametric space or a self-similar set with the strong separation condition (SSC) and let $B$ be a compact metric space such that $\hdim A > \ubdim B$. Then $A$ can be mapped onto $B$ by a Lipschitz map.
\end{cor}

In Section~\ref{s:genlip} we also prove a  result (Proposition~\ref{p:ultraholder} \eqref{i:lip}), which is slightly stronger than Corollary~\ref{c:dimwithextension}: it can be also applied for some cases when $\hdim A=\ubdim B$.
To give a partial answer to Question~\ref{q:intro} (b) we give a more explicit sufficient condition (Corollary~\ref{c:Cantor}) for the case when $A$ is the middle third Cantor set.

In many cases, for example if $A$ is connected and $B$ is non-connected, we clearly cannot have a Lipschitz map from $A$ onto $B$ but we might hope for an extension to a space that contains $B$, obtaining a cover of $B$ by a Lipschitz image of $A$, similarly as in the classical analyst's traveling salesman theorem. It is easy to prove that any
real valued Lipschitz function defined on a subset of a metric space $X$ has a Lipschitz extension to the whole space $X$, see e.g.~\cite[p.~202]{Fe}. Applying this extension in each coordinate, we obtain that every Lipschitz function from a subset of a metric space $X$ to $\R^n$ has a Lipschitz extension to the whole $X$. Using this and Kirszbraun's extension theorem \cite{Kir}, Theorem~\ref{t:dimimplieslip} yields the next corollary.

\begin{cor}
Suppose that $X$ and $Y$ are Hilbert spaces or $X$ is an arbitrary metric space and $Y=\R^n$ for some $n$.
Let $A$ be a compact subset of $X$ and $B$ be a compact subset of $Y$ such that $\hdim A > \ubdim B$. Then $B$ can be covered by a Lipschitz image of $A$, that is, there is a Lipschitz map $f\colon A\to Y$ with $B\subset f(A)$.
\end{cor}


From Theorem~\ref{t:dimimplieslip} we also deduce the following result, which might be
interesting in itself.
It states that any `natural' concept of dimension must 
be in between the Hausdorff and the upper box dimensions if we do not assume $\sigma$-stability, and
in between the Hausdorff and the packing dimensions
($\pdim$, see the definition in Section~\ref{s:prel}), if we require $\sigma$-stability.  
Here we state it only for dimensions defined on compact sets; in Section~\ref{s:dim} we prove a stronger statement for more general domains.

\begin{cor}\label{c:dim}
Let $n$ be a positive integer and let $D$ be a function from the family of compact subsets of $\R^n$ to $[0,n]$. Suppose that 
\begin{enumerate}[(i)] 
\item Lipschitz functions cannot increase $D$, that is, for any compact $K\su\R^n$ and Lipschitz function $f\colon K\to\R^n$ we have $D(f(K))\le D(K)$, 
\item  $D$ is monotone, that is, $A\subset B$ implies that $D(A)\le D(B)$, and 
\item if $S\su\R^n$ is a homogeneous self-similar set with the SSC then $D(S)=\hdim S$.
\end{enumerate}
Then for every compact set $K\su \R^n$ we have \begin{equation*} \hdim(K)\le D(K)\le \ubdim(K).
\end{equation*}
Furthermore, if we also require that
$D$ is $\sigma$-stable, that is,
\begin{enumerate}[(i)]
\setcounter{enumi}{3}
\item 
$D(\cup_{i=1}^{\infty} K_i)=\sup_{i} D(K_i)$ whenever
$\cup_{i=1}^{\infty} K_i$ and all $K_i$ are compact, 
\end{enumerate}
then for every compact set $K\su \R^n$ we have
\begin{equation*}
\hdim(K)\le D(K)\le \pdim(K).
\end{equation*}
\end{cor}

The above result sheds some light on why the intermediate dimensions introduced by Falconer, Fraser and Kempton \cite{FFK} and the generalized intermediate dimensions introduced by Banaji \cite{Ban} span between the Hausdorff and upper box dimensions, and their $\sigma$-stable versions \cite{DS} are between the Hausdorff and packing dimensions.

Note that there are other fractal dimensions, such as the Assouad dimension, which do not lie between the Hausdorff and the upper box dimensions. However, this does not contradict Corollary~\ref{c:dim},
since these dimensions do not satisfy all the assumptions; for example, Lipschitz maps can increase the Assouad dimension.

\subsection{H\"older images of compact metric spaces}
In order to prove Theorem~\ref{t:dimimplieslip} we need to consider H\"older images of $[0,1]$. 

\begin{question}
For given $\alpha>0$ which compact metric spaces $B$ can be obtained as an $\alpha$-H\"older image of a compact subset of $[0,1]$? 
\end{question}

This is basically the H\"older version of the analyst's traveling salesman problem, for which M.~Badger, L.~Naples, and V.~Vellis \cite{BNV} found a deep sufficient condition if $\alpha<1$ and $B$ is a subset of a Euclidean or Hilbert space. We provide a different type of characterization in Theorem~\ref{t:delta1}.

\begin{defi}\label{d:delta}
Let $(X,d)$ be a metric space
and let $n$ be a positive integer. 
By $\Sym(n)$ we denote the set of
$\{1,\ldots,n\}\to\{1,\ldots,n\}$ bijections.
For $x_1,\ldots,x_n\in X$ and $s>0$ define 
\begin{equation*}
Z^s(x_1,\ldots,x_n)=\max\left\{\sum_{j=1}^{k-1} (d(x_{i_j},x_{i_{j+1}}))^s: 1=i_1<\ldots<i_k=n\right\},
\end{equation*}
\begin{equation*}
\delta^s(X)=\sup\left\{\min_{\pi\in \Sym(n)} Z^s(x_{\pi(1)},\dots,x_{\pi(n)}):~x_1,\dots,x_n \in X,\,n\geq 1\right\}.
\end{equation*}
\end{defi} 

\begin{theorem}
\label{t:delta1}
Let $X$ be a compact metric space and $s>0$. 
Then there exist a compact set $D\subset [0,1]$ and a $1/s$-H\"older onto map $g\colon D\to X$ if and only if $\delta^{s}(X)<\infty$.
\end{theorem}

\begin{remark}
It is well known that if $\alpha>1$ then every $\alpha$-H\"older function $f\colon [0,1]\to \R^n$ is constant. Indeed, on one hand $f([0,1])$ is connected, on the other hand its Hausdorff dimension is at most $1/\alpha<1$. This happens only because $[0,1]$ is connected, which is illustrated by the following example. Let $0<\beta<\gamma<1$ be arbitrary and let $C_{\beta}$ and $C_{\gamma}$ be the homogeneous self-similar Cantor sets of dimensions $\beta$ and $\gamma$, respectively. Then it is easy to see that the natural bijection $\varphi \colon C_{\gamma}\to C_{\beta}$ is $\gamma/\beta$-H\"older. 
\end{remark}

The next theorem connects $\delta^s$ to the upper box dimension. 

\begin{theorem}
\label{t:delta2}
If $X$ is a compact metric space with $\ubdim X<s$, then $\delta^s (X)<\infty$.

Moreover, for any compact metric space $X$ we have
\begin{equation*}
\ubdim X = \inf\{s>0\colon \delta^s(X)<\infty\}.
\end{equation*}
\end{theorem}

In fact, we prove both Theorems~\ref{t:delta1} and \ref{t:delta2} in a slightly stronger, quantitative form in Section~\ref{s:holder}. Combining Theorems~\ref{t:delta1} and \ref{t:delta2} immediately gives the following corollary.

\begin{cor} \label{c:Holder}
If $\al>0$ and $B$ is a compact metric space with $\ubdim B<1/\alpha$ then $B$ can be obtained as the $\alpha$-H\"older image of a compact subset of $[0,1]$. 
\end{cor}

For the special case when $B\subset \RR^m$ and $\alpha\leq 1$, Corollary~\ref{c:Holder} is essentially known: it follows easily
from \cite[Theorem~2.3]{BV} or the more general \cite[Theorem~1.1]{BNV}. 

By an extension theorem of Minty \cite[Theorem~1~(ii)]{Min}, if $\alpha\le 1$, $D\subset [0,1]$ and $H$ is a Hilbert space then any $\alpha$-H\"older map $g\colon D\to H$ can be extended to $[0,1]$ as an $\alpha$-H\"older map.
Thus Corollary~\ref{c:Holder} has the
following direct consequence.


\begin{cor} If $B$ is a compact subset of a Hilbert space and $\ubdim B<1/\alpha$ and $\alpha\leq 1$ then $B$ can be covered by an $\alpha$-H\"older image of $[0,1]$.
\end{cor}

\subsection{Self-similar sets with the strong separation condition}

Let $A\su\R^n$ and $B\su\R^m$ be self-similar sets with the strong separation condition (SSC).
It is well known (see e.g. in \cite{Fa}) that for such sets all dimensions agree, including Hausdorff dimension and upper box dimension.
Since Lipschitz maps cannot increase Hausdorff dimension, $A$ cannot be mapped onto $B$ by a Lipschitz map if $\dim A< \dim B$.
On the other hand, by a special case of Corollary~\ref{c:dimwithextension},  $A$ can be mapped onto $B$ by a Lipschitz map if $\dim A>\dim B$. 
For this special case 
Deng, Weng, Xiong and Xi \cite{DWXX} proved a stronger result by showing that in this case $\dim A>\dim B$ implies that there is a bilipschitz embedding $f\colon B\to A$. 
(Note that by Lemma~\ref{l:extension}, the inverse of $f$ can be extended to $A$, so 
the existence of a bilipschitz embedding $f\colon B\to A$ indeed implies that $A$ can be mapped onto $B$ by a Lipschitz map.) 

So it remains to study the $\dim A=\dim B$ case.
In \cite{DWXX} Deng, Weng, Xiong and Xi also proved that if $A$ and $B$ are self-similar sets with the SSC and $\dim A=\dim B$ then $B$ can be bilipschitz embedded to $A$ if and only if $A$ and $B$ are bilipschitz equivalent.

There is a vast literature establishing conditions under which two metric spaces, or more specifically two self-similar sets are bilipschitz equivalent, see e.g \cite{FM, LL, X, XX} and the references therein. 
In 1992 Falconer and Marsh~\cite{FM} found necessary algebraic conditions for two self-similar sets with the SSC to be bilipschitz equivalent. In 2010 Xi~\cite{X} gave a necessary and sufficient condition.
For the case when one of the sets is the middle third Cantor set $C$, the Falconer-Marsh necessary condition is also sufficient and it is very simple: 
a self-similar set $B$ with the SSC and with $\dim B=\dim C$ is bilipschitz equivalent to $C$ if and only if all similarity ratios of $B$ are integer powers of $3$.

We prove that this condition is necessary even to have a Lipschitz onto map from $C$ onto $B$. 
Therefore we can answer a special case of Question~\ref{q:intro}: a self-similar set $B$ with the SSC and $\dim B=\dim C$ can be obtained as a Lipschitz image of the Cantor set $C$ if and only if 
$B$ and $C$ are bilipschitz equivalent.
In Section~\ref{s:ssc} we prove this not only for the Cantor set but also 
for any homogeneous self-similar set with the SSC:

\begin{theorem}
\label{t:ssc}
Let $A$ and $B$ be self-similar sets with the SSC such that $A$ is homogeneous. Suppose that $\hdim A=\hdim B$. Then $A$ can be mapped onto $B$ by a Lipschitz map 
if and only if $A$ and $B$ are bilipschitz equivalent.
\end{theorem}

\subsection{Structure of the paper}
The paper is organized as follows.
Section~\ref{s:prel} contains some basic definitions and results we use later. 
We prove Theorems~\ref{t:delta1} and \ref{t:delta2} in somewhat stronger forms in Section~\ref{s:holder},
Theorem~\ref{t:dimimplieslip} and some related results in Section~\ref{s:genlip},
Lemma~\ref{l:extension} in Section~\ref{s:extension}, 
Corollary~\ref{c:dim} in a stronger form in Section~\ref{s:dim} 
and Theorem~\ref{t:ssc} in a stronger form in Section~\ref{s:ssc}.
Finally, in Section~\ref{s:questions} we pose some questions.

\section{Preliminaries}\label{s:prel}

Let $(X,d)$ be a metric space. For $A,B \subset X$ let $\dist(A,B) = \inf\{d(x,y) : x\in A,~y\in B\}$. For $x\in X$ and $r>0$ we denote by $\B(x,r)$ the closed ball of radius $r$ centered at $x$, and by $N(X,r)$ the minimal number
of closed balls of radius $r$ that cover $X$. 
The \emph{upper box dimension} of a bounded set $X\su\R^n$ is defined as
\begin{equation*} 
\ubdim X=\limsup_{r \to 0+} \frac{\log N(X,r)}{\log (1/r)}.
\end{equation*}
The \emph{packing dimension} of a set $X\su\R^n$ can be defined as
the $\sigma$-stable modification of upper box dimension:
\begin{equation}\label{e:pdim}
    \pdim X=\inf\left\{\sup_i \ubdim X_i \colon X=\cup_{i=1}^\infty X_i, X_i \textrm{ is bounded}\right\}.
\end{equation}
For more on these dimensions and for the concepts of the \emph{Hausdorff dimension} $\dim_H$ and \emph{$s$-dimensional Hausdorff measure} $\iH^s$ see e.g.~the books \cite{Fa} and \cite{Ma}. 

A metric space $(X,d)$ is called \emph{ultrametric} if the triangle inequality is replaced with the stronger inequality
\begin{equation*} 
d(x,y)\leq \max\{d(x,z),d(y,z)\} \quad  \textrm{for all } x,y,z\in X.
\end{equation*}

This is equivalent to the property that if $a\leq b\leq c$ are sides of a triangle in $X$, then $b=c$. 
The following useful fact follows easily from the definition.

\begin{fff} \label{f:u} Let $(X,d)$ be an ultrametric space. For all $x,y\in X$ and $r>0$ either $\B(x,r)\cap \B(y,r)=\emptyset$ or $\B(x,r)=\B(y,r)$.
\end{fff}

Suppose that $K\subset \RR^n$ is a self-similar set with contracting similarity maps $\{f_i\}_{1\leq i\leq m}$; that is $K=\cup_{i=1}^m f_i(K)$.
We say that $K$ satisfies the \emph{strong separation condition} (SSC) if $f_i(K)\cap f_j(K)=\emptyset$ for all $1\leq i< j\leq m$. We say that $K$ is \emph{homogeneous} if the similarity ratios
of all the similarity maps $f_i$ are the same.


\bigskip

Let $(X,d_X)$ and $(Y,d_Y)$ be metric spaces and $f\colon X\to Y$. We say that $f$ is \emph{$\alpha$-H\"older} if there exists a finite constant $C$ such that for all $x,z\in X$ we have 
\begin{equation*} d_Y(f(x),f(z))\leq C d_X(x,z)^{\al}. 
\end{equation*} 
If we can take $C=1$ in the above equation, then $f$ is called \emph{$\alpha$-$1$-H\"older}. The $1$-H\"older functions are also called \emph{Lipschitz} functions, and the $1$-$1$-H\"older functions are called \emph{Lipschitz-$1$} functions.

The metric spaces $X$ and $Y$ are said to be \emph{bilipschitz equivalent} if there exists a bijection $f\colon X\to Y$ such that both $f$ and its
inverse are Lipschitz.
\bigskip 

A subset $A$ of a separable complete metric space $X$ is called \emph{analytic} if it can be obtained as a continuous image of a complete separable metric space $Y$. 
It is well known that all Borel sets are analytic, see e.g.~\cite{Kec}.

\bigskip

For completeness we provide a proof for the following useful fact.

\begin{lemma}\label{sscbilipultra}
    Let $K\su\R^n$ be a self-similar set obtained from the similarity maps
    $f_1,\ldots,f_m$ with similarity ratios $r_1,\ldots,r_m$, respectively. 
    If $K$ satisfies the strong separation condition then it is bilipshitz equivalent to an ultrametric space $(Y,d)$, where $(Y,d)$ depends only on the similarity ratios $r_1,\ldots,r_m$.
\end{lemma}

\begin{proof}
%
  Let $Y=\{1,\ldots,m\}^{\N}$ with the metric $d((i_1, i_2, \ldots), (j_1, j_2, \ldots))=r_{i_1}\cdot\ldots\cdot r_{i_k}$ whenever $i_1=j_1, \ldots, i_k=j_k$ but $i_{k+1}\neq j_{k+1}$. 
  It is easy to check that $(Y,d)$ is an ultrametric space and it clearly depends only on $r_1,\ldots,r_m$. 
  It is also easy to check that the function 
  $f((i_1,i_2,\ldots))=\cap_{k=1}^\infty f_{i_1}\circ\ldots\circ f_{i_k}(K)$ is
  well defined and gives a bilipschitz equivalence between $(Y,d)$ and $K$.
\end{proof}

\section{H\"older images of compact subsets of the real line}
\label{s:holder}

The goal of this section is to prove Theorems~\ref{t:delta1} and \ref{t:delta2}. First we need two lemmas. 
The proof of the first one is a standard compactness argument.

\begin{lemma}\label{l:finitetocompact}
Let $(X,d)$ be a compact metric space and $\al, \ell>0$. 
Then there exist a compact set $D\subset [0,\ell]$ and an $\alpha$-$1$-H\"older onto map $g\colon D\to X$ if and only if for any finite subset $V\su X$ there exists a finite
set $W\su [0,\ell]$ and an $\alpha$-$1$-H\"older onto map
$h\colon W\to V$. 
\end{lemma}

\begin{proof}
One implication is clear, we prove the other one.
Let $B=\{b_1,b_2,\ldots\}$ be a dense subset of $X$.
By the assumption of the lemma for every $n$ there exist
$W_n=\{t_{n,1},\ldots t_{n,n}\}\su [0,\ell]$ and an $\alpha$-$1$-H\"older
$h_n\colon W_n\to X$ such that $h_n(t_{n,i})=b_i$ 
for every $i=1,\ldots,n$.
By taking convergent subsequences,
we can get a nested sequence of infinite index sets
$\N\supset I_1\supset I_2\supset\ldots$ such that for every
$k$ the subsequence $(t_{n,k})_{n\in I_k}$ converges.
Let $a_k$ be the limit and let $g(a_k)=b_k$.

We claim that $g$ is $\alpha$-$1$-H\"older on $A=\{a_1,a_2,\ldots\}$.
As $h_n$ is $\alpha$-$1$-H\"older, we have $d(b_i,b_j)\le |t_{n,i}-t_{n,j}|^\al$ for all $n$.
Fix $i<j$. Since $I_j\su I_i$ we get that 
$(t_{n,i})_{n\in I_j}\to a_i$ and $(t_{n,j})_{n\in I_j}\to a_j$. Therefore, $d(g(a_i),g(a_j))=d(b_i,b_j)\le |a_i-a_j|^\al$, so $g$ is indeed
$\alpha$-$1$-H\"older on $A$.

Let $D$ be the closure of $\{a_1,a_2,\ldots\}$.
Then $D\su[0,\ell]$ is compact and $g$ clearly extends to $D$ as an $\alpha$-$1$-H\"older function.
Since the compact set $g(D)$ contains the dense set $B$, we have $g(D)=X$, that is, $g\colon D\to X$ is onto.
\end{proof}

\begin{lemma}\label{l:Z} For any metric space $(X,d)$, $1\le i<j$, $x_1,\ldots,x_j\in X$ and $s>0$
we have the following inequalities 
(recall Definition~\ref{d:delta}):
\begin{equation}\label{e:Z1}
    Z^s(x_1,\ldots,x_i)+Z^s(x_i,\ldots,x_j)\le Z^s(x_1,\ldots,x_j),
\end{equation}
 \begin{equation}\label{e:Z2}
    Z^s(x_1,\ldots,x_j)-Z^s(x_1,\ldots,x_i)\ge Z^s(x_i,\ldots,x_j)\ge (d(x_i,x_j))^s.
\end{equation}   
\end{lemma}

\begin{proof}
The inequality \eqref{e:Z1} and the second part of \eqref{e:Z2} follows from the definition, the first part of \eqref{e:Z2} clearly follows from \eqref{e:Z1}.
\end{proof}

Now we can prove Theorem~\ref{t:delta1}.
In fact, we prove a bit more.

\begin{theorem}\label{t:strongerdelta1}
Let $X$ be a compact metric space and $s>0$. Then there exist a compact set $D\subset [0,1]$ and a $1/s$-H\"older onto map $g \colon D\to X$ if and only if $\delta^{s}(X)<\infty$.

Moreover, if $\delta^{s}(X)<\infty$, then
\begin{equation*}
   \delta^{s}(X)=\min\{\ell>0 : \exists  D\subset [0,\ell] \text{ compact and }
   g\colon D\to X \text{ $(1/s)$-$1$-H\"older onto}\}.
\end{equation*}
\end{theorem}

\begin{proof}
First we prove that the existence of a compact set $D\su [0,\ell]$ and a $(1/s)$-$1$-H\"older onto map $g \colon D\to X$ implies
that $\delta^{s}(X)\le \ell$.    
Let $x_1,\ldots,x_n\in X$ be arbitrary and pick $t_i\in g^{-1}(\{x_i\})$.
There exists a permutation $\pi\in \Sym(n)$ such that $t_{\pi(1)}\le\ldots\le t_{\pi(n)}$. 
Since $g$ is $(1/s)$-$1$-H\"older, we obtain 
\begin{equation*} 
d(x_{\pi(i)},x_{\pi(i')})\le |t_{\pi(i)}-t_{\pi(i')}|^{1/s}
\end{equation*} 
for any $i, i'$.
Thus for any $1=i_1<\ldots<i_k=n$ we have 
\begin{equation*}
\sum_{j=1}^{k-1} (d(x_{\pi(i_j)},x_{\pi(i_{j+1})}))^{s}\le
\sum_{j=1}^{k-1} |t_{\pi(i_j)}-t_{\pi(i_{j+1})}|\le \ell.
\end{equation*}
Therefore $Z^{s}(x_{\pi(1)},\ldots,x_{\pi(n)})\le \ell$, which implies that $\delta^{s}(X)\le \ell$.

Finally, we show that if $\delta^{s}(X)<\infty$ then there
exist a compact set $D\subset [0,\delta^{s}(X)]$ and a
$(1/s)$-$1$-H\"older onto map $g \colon D\to X$.
Let $V=\{x_1,\ldots,x_n\}\su X$ be an arbitrary finite subset of $X$ enumerated such that 
$Z^{s}(x_1,\ldots,x_n)\le \delta^{s}(X)$.
For each $i=1,\ldots,n$ define $a_i=Z^{s}(x_1,\ldots,x_i)$, 
$h(a_i)=x_i$ and $W=\{a_1,\ldots,a_n\}$.
Clearly, $W\su[0,\delta^{s}(X)]$ is finite, and by \eqref{e:Z2} of Lemma~\ref{l:Z} we obtain that $h\colon W\to V$ is a
$(1/s)$-$1$-H\"older onto map. 
Since $V\su X$ was an arbitrary finite subset, by Lemma~\ref{l:finitetocompact}
the proof is complete.
\end{proof}

The following theorem clearly contains Theorem~\ref{t:delta2}. For the notation $N(X,r)$ recall Section~\ref{s:prel}.

\begin{theorem}\label{t:strongerdelta2}
For any compact metric space $(X,d)$ the following statements hold. 
\begin{enumerate}
\item\label{i:deltasandN}
$\displaystyle{
\delta^s(X)\le \left(\frac{2\diam X}{u(1-u)}\right)^s\sum_{n=1}^\infty N(X,u^{n}) \cdot u^{ns} \quad \text{for all } s>0 \text{ and } 0<u<1;\displaystyle}$ \bigskip
\item \label{i:finitedelta}
If $s>\ubdim X$ then $\delta^s(X)<\infty$; \bigskip
\item \label{i:bdimdelta}
$\displaystyle{
\ubdim X = \inf\{s>0\colon \delta^s(X)<\infty\};
\displaystyle}$ \bigskip 
\item \label{i:bdimholder}
$\displaystyle{
\ubdim X = \inf\{s>0\colon \exists D\subset[0,1] \text{ and } 1/s\text{-H\"older onto } g\colon D\to X\}.
\displaystyle}$
\end{enumerate}

\end{theorem}

\begin{proof}
It is easy to see that, by the definition of the upper box dimension, \eqref{i:deltasandN} implies \eqref{i:finitedelta}.
It is clear that \eqref{i:finitedelta} implies
`$\ge$' in \eqref{i:bdimdelta}.
The other inequality in \eqref{i:bdimdelta}
follows from the observation that $\ubdim X>s$ implies that there exists $\eps$-net
of large cardinality in $X$, which clearly
implies that $\delta^s(X)$ is large.
By Theorem~\ref{t:strongerdelta1}, statements \eqref{i:bdimdelta} and \eqref{i:bdimholder} are clearly equivalent.
Therefore it remains to prove \eqref{i:deltasandN}.


By rescaling $X$, we can suppose that $\diam X=1$.
Let $a_n= N(X,u^n)$ and let $H_n$ be the set of centers of a collection of $a_n$ balls of radius $u^n$ that covers $X$. Since $\diam X= 1$, we have $a_0=1$.

We define an infinite tree such that for each $n\in \N$ the vertices of the $n$th level are the points of $H_n$, and let us endow $H_n$ with an arbitrary ordering $(H_n, <_n)$. Since $\cup_{h\in H_n} \B(h,u^n)=X$, for every $h'\in H_{n+1}$ there exists an $h\in H_n$ with $d(h,h')\le u^n$. Join $h'$ by an edge to exactly one such $h\in H_n$. Let $T$ be the tree we obtained. Consider the infinite branches of $T$. By the compactness of $X$ and since $u<1$ the vertices of every infinite branch $(h_n)_{n\geq 0}$ of $T$ converge to a point of $X$.

We claim that the converse also holds: for any $x\in X$ there exists an infinite branch $(h_n(x))_{n\geq 0}$ of $T$ that converges to $x$. Indeed, let $W_x=\cup_{n=0}^\infty\{h\in H_n \colon x\in \B(h,u^n)\}$ and let $V_x$ consists of $W_x$ and all the ancestors of all points of $W_x$. Then the points of $V_x$ form a subtree of $T$ with arbitrarily long branches. Since every degree of $T$ is finite, by K\"onig's lemma, this implies that the subtree $V_x$ has an infinite branch $(h_n(x))_{n\geq 0}$. Let $y$ be the limit of $h_n(x)$. Since by definition $h_n(x)\in W_x$ for infinitely many $n$, this implies that $y=x$, which completes the proof of the claim.

Using the orderings $(H_n, <_n)$,
the lexicographic ordering gives an order $\prec$ on $X$: let $x\prec y$ if for some $n\geq 1$ we have $h_0(x)=h_0(y),\ldots,h_{n-1}(x)=h_{n-1}(y)$ and $h_n(x)<_n h_n(y)$. Let $y_1,\ldots,y_k\in X$ be arbitrarily fixed.  
To prove \eqref{i:deltasandN}, and so also the theorem,
it is enough to show that
\begin{equation}\label{e:enough}
y_1\prec \ldots \prec y_k \Longrightarrow
\sum_{j=1}^{k-1} (d(y_j,y_{j+1}))^s \le\left(\frac{2}{u(1-u)}\right)^s \sum_{n=1}^\infty a_n\cdot u^{ns}.  
\end{equation} 
Assume $y_1 \prec \ldots \prec y_k$. We partition $\{y_1,\ldots,y_{k-1}\}$ as follows: for all $n\geq 1$ define
\begin{equation*}
Y_n=\left\{ y_j \colon j \leq k-1,\, h_i(y_j)=h_i(y_{j+1}) \textrm{ for all }  0\leq i<n \textrm{ and } 
h_n(y_j)<h_n(y_{j+1})\right\}.
\end{equation*}
Note that $h_n\colon (Y_n,\prec)\to (H_n,<_n)$ is a strictly increasing map, so $|Y_n|\le |H_n| = a_n$. 

For each $x\in X$ we have $d(h_{i}(x),h_{i+1}(x))\le u^{i}$ and $h_i(x)\to x$ as $i\to \infty$, so for all $n\geq 1$ we obtain 
\begin{equation*} 
d(h_{n-1}(x),x)\le \sum_{i=n-1}^\infty u^i = \frac{u^{n-1}}{1-u}.
\end{equation*}

Suppose that $y_j\in Y_n$. 
Then $h_{n-1}(y_j)=h_{n-1}(y_{j+1})$ and we obtain
\begin{equation*}
d(y_j,y_{j+1})\le d(y_j,h_{n-1}(y_j))+d(y_{j+1},h_{n-1}(y_{j+1}))
\le \frac{2u^{n-1}}{1-u}.    
\end{equation*}
Therefore
\begin{align*}
\begin{split}
\sum_{j=1}^{k-1} (d(y_j,y_{j+1}))^s &=
\sum_{n=1}^\infty\sum_{y_j\in Y_n} (d(y_j,y_{j+1}))^s 
\\ 
&\le \sum_{n=1}^\infty |Y_n|\cdot 
\left(\frac{2u^{n-1}}{1-u}\right)^s \\
&=\left(\frac{2}{u(1-u)}\right)^s \sum_{n=1}^\infty a_n u^{ns},
\end{split}
\end{align*}
so \eqref{e:enough} holds. This completes the proof of the theorem.
\end{proof}

\begin{remark}\label{r:CantorasHolder}
For $X\su\R$ it is easier to calculate $\delta^s(X)$ since in this case in  Definition~\ref{d:delta} it is clear which permutation $\pi\in\Sym(n)$ gives the minimum.
For example, consider the middle third Cantor set $C$ and let $s=\log 2/ \log 3$, then one can easily see that $\delta^s(C)=\infty$.
Thus, by Theorem~\ref{t:strongerdelta1}, $C$ cannot be obtained as a $1/s$-H\"older image of a compact subset of $\R$. 
On the other hand, by Theorem~\ref{t:strongerdelta2} and Corollary~\ref{c:Holder}, for any $t>\ubdim C=s$ we have $\delta^t(C)<\infty$  and $C$ can be obtained as a $1/t$-H\"older image of a compact subset of $\R$.
\end{remark}

\section{H\"older images of compact metric spaces, Lipschitz images of the Cantor set}
\label{s:genlip}

The goal of this section is to prove Theorem~\ref{t:dimimplieslip} (for H\"older maps) and some related results.
We combine results from Section~\ref{s:holder} with arguments and results from \cite{KMZ}. First we prove the following statement, which is in fact a stronger version of Corollary~\ref{c:dimwithextension}.

\begin{prop}\label{p:ultraholder}
    Let $A$ be a compact ultrametric space
    with positive $t$-dimensional Hausdorff measure and 
    $B$ be a compact metric space with $\delta^s(B)<\infty$.
\begin{enumerate}
    \item \label{i:tpersholder}
     There exists a compact $A'\subset A$ and a $t/s$-H\"older onto map $f\colon A'\to B$.
     \item \label{i:lip}
     If $t=s$ then $A$ can be mapped onto $B$ by a Lipschitz map.
\end{enumerate}
\end{prop}

\begin{proof}
\eqref{i:tpersholder}: By \cite[Theorem~2.1, Lemma 2.3]{KMZ} any compact ultrametric space
with positive $t$-dimensional Hausdorff measure can be mapped onto $[0,1]$ by a $t$-H\"older function.
Hence there exists a $t$-H\"older onto map $h\colon A\to [0,1]$.
By Theorem~\ref{t:delta1}, there exist a compact $D\subset [0,1]$ and a $1/s$-H\"older onto map $g\colon D\to B$. 
Let $A'=h^{-1}(D)$. Then $A'\su A$ is clearly compact and $g\circ h\colon A'\to B$ is a $t/s$-H\"older onto map.

\eqref{i:lip}: By \eqref{i:tpersholder} there exists an $A'\subset A$ and a Lipschitz onto map $f\colon A'\to B$.
Lemma~\ref{l:extension} implies that $f$ can be extended to $A$ as a Lipschitz map.
\end{proof}

\begin{remark}
Recall that the middle third Cantor set $C$ is bilipschitz equivalent to an ultrametric space. Let $s=\log 2/\log 3$. Proposition~\ref{p:ultraholder}\eqref{i:lip} gives a sufficient condition for Question~\ref{q:intro}(b):
a compact metric space $B$ can be covered by a Lipschitz image of $C$ if $\delta^{s}(B)<\infty$.
On the other hand, as we saw in Remark~\ref{r:CantorasHolder}, $\delta^{s}(C)=\infty$, so the $B=C$ example shows that $\delta^{s}(B)<\infty$ is not a necessary condition for Question~\ref{q:intro} (b).
\end{remark}

Theorem~\ref{t:dimimplieslip} is clearly the $\alpha=1$ special case of the following theorem.

\begin{theorem}\label{t:dimimpliesholder}
 Let $A$ and $B$ be compact metric spaces such that $\hdim A > \alpha \ubdim B$ for
 some $\alpha>0$.
Then there exists a compact set $A'\su A$ and an $\alpha$-H\"older onto map $f\colon A'\to B$.   
\end{theorem}

\begin{proof}
Let $t\in (\alpha \ubdim B, \hdim A)$.
By a deep theorem of Mendel and Naor~\cite{MN}, $\hdim A>t$ implies that there exists a compact set $A'\su A$
such that $\hdim A'>t$ and $A'$ is 
bilipschitz equivalent to an ultrametric
space. 
Since $\ubdim B<t/\alpha$, 
Theorem~\ref{t:delta2} implies that $\delta^{t/\alpha}(B)<\infty$.
Then Proposition~\ref{p:ultraholder} \eqref{i:tpersholder} implies that there is an $\alpha$-H\"older onto map $f\colon A'\to B$.   
\end{proof}





\begin{remark}
By Howroyd's theorem \cite{Ho}, if $A$ is an analytic subset of a separable complete metric space $X$ with $\hdim A>s$ for some $s$ then there exists a compact set $A'\su A$ with
$\hdim A'>s$. Thus in Theorem~\ref{t:dimimpliesholder}
(and so also in Theorem~\ref{t:dimimplieslip})
$A$ does not have to be compact, it is enough to assume that $A$ is an analytic subset of a separable complete metric space.
\end{remark}

The following result gives partial answer to 
Question~\ref{q:intro} (b).
We state the essentially trivial necessary condition \eqref{i:necessery} just to show that the obtained necessary and sufficient conditions are not very far from each other.
Note that the sufficient condition in \eqref{i:sufficient} is weaker than the condition $\ubdim B<\log 2/\log3$ in Corollary~\ref{c:dimwithextension}:
it also allows some compact metric spaces with $\ubdim B=\log 2/\log 3$.

\begin{cor}\label{c:Cantor}
    Let $C$ be the middle third Cantor set and let $B$ be an arbitrary compact metric space. 
    Let $b_n=N(B,3^{-n})$.
\begin{enumerate}
    \item \label{i:sufficient}
    If $\sum_{k=1}^\infty b_n/2^n$ converges then $C$ can be mapped onto $B$ by a Lipschitz map.
    \item \label{i:necessery}
    If $C$ can be mapped onto $B$ by a Lipschitz map then the sequence $b_n/2^n$ must be bounded.
\end{enumerate}    
\end{cor}

\begin{proof}
Let $s=\log 2/\log 3$. Applying Theorem~\ref{t:strongerdelta2} \eqref{i:deltasandN} for $u=1/3$ implies that $\delta^s(B)<\infty$. Since $C$ is bilipschitz equivalent to an ultrametric space and its $s$-dimensional Hausdorff measure is positive, 
Proposition~\ref{p:ultraholder}\eqref{i:lip} completes the proof of \eqref{i:sufficient}. The necessary condition \eqref{i:necessery} easily follows from $N(C,3^{-n})\le 2^n$.
\end{proof}

\begin{remark}
Neither the condition of \eqref{i:sufficient}, nor the condition of \eqref{i:necessery} can be necessary and sufficient. 
If $B=C$ then $\sum b_n/2^n$ clearly diverges but the identity map is a trivial Lipschitz onto map. On the other hand, we show a compact set $B\subset \R^2$ such that $b_n/2^n$ is bounded but for $s=\log_3 2$ the $s$-dimensional packing measure of $B$ is infinite. Since $C$ clearly has finite $s$-dimensional packing measure, and the packing measure of a Lipschitz image is at most a finite multiple of the packing measure of the original set, this indeed implies that there is no Lipschitz onto map $f\colon C\to B$. Consider a Bedford--McMullen carpet $B\subset \R^2$ (see \cite{Be} or \cite{Mc}) with $m=3^2$ rows and $n=3^4$ columns, such that the pattern $D\subset \{0,\dots,n-1\}\times \{0,\dots, m-1\}$ has $8$ elements and $3$ are in one row and $5$ are in another one. Then the projection to the second coordinate, $\pi(D)$, has $2$ elements. 
Then by \cite{Mc} we have the formula
\begin{equation*}
\ubdim B=\log_m |\pi(D)|+\log_n \frac{|D|}{|\pi(D)|}=\log_3 2,
\end{equation*} 
and an easy calculation also yields that $b_n/2^n$ is bounded. On the other hand, as the $2$ non-empty rows of $D$ have different cardinalities, the $s$-dimensional packing measure of $B$ is infinite by \cite[Theorem~1.1]{Pe}.
\end{remark}

\section{Extensions, proof of Lemma~\ref{l:extension}}
\label{s:extension}

In this section we prove Lemma~\ref{l:extension}. 
A usual way to prove an extension result like Lemma~\ref{l:extension} is to prove a statement about retracts. 
In this case we need the following.

\begin{lemma}\label{l:lipretract}
Let $(X,d)$ be a compact ultrametric space and let $A\su X$ be a compact subset. Then there exists a Lipschitz-$1$ retraction $g\colon X\to A$, that is, a function such that $g(x)=x$ for any $x\in A$.
\end{lemma}

\begin{proof}
By a \emph{sphere} we mean a set of the form $S(x,r)=\{y\in X : d(x,y)=r\}$, where $x\in X$ and $r>0$. From each sphere $S$ such that $S\cap A\neq\emptyset$
we choose a point $p(S)\in S\cap A$.
For $x\in A$ let $g(x)=x$, for $x\in X\sm A$ let
$g(x)=p(S(x,\dist(x,A)))$.

It remains to check that $g$ is Lipschitz-$1$. Let $x,y\in X$ be given, we need to prove that $d(g(x),g(y)) \leq d(x,y)$. This is clear if $x,y\in A$. 

Now assume $x\in A$, $y\not\in A$.
By definition $d(y,g(y))=\dist(y,A)\le d(y,x)$.
Hence the ultrametric property for the triangle with vertices $x$, $y$, $g(y)$ yields $d(g(y),g(x))=d(g(y),x)\leq d(y,x)$.

Finally, suppose $x,y\not\in A$. Assume to the contrary that $d(x,y)<d(g(x),g(y))$. The definition of $g$ implies that 
\begin{equation} \label{e:um} 
d(x,g(x))=\dist(x,A)\le d(x,g(y))  \textrm{ and }  d(y,g(y))=\dist(y,A)\le d(y, g(x)).
\end{equation}
Applying \eqref{e:um} and the ultrametric property for the triangles with vertices $x$, $g(x)$, $g(y)$, and $y$, $g(y)$, $g(x)$, respectively, we obtain
\begin{equation} \label{e:um2} 
d(g(x),g(y))\le d(x,g(y)) \quad \textrm{and} \quad d(g(x),g(y))\le d(y, g(x)).
\end{equation}
Then $d(x,y)<d(g(x),g(y))$ and inequalities \eqref{e:um2} imply 
$d(x,y)<d(x,g(y))$ and $d(x,y)<d(y, g(x))$. Hence the ultrametric property for the triangles with vertices $x$, $y$, $g(y)$, and $x$, $y$, $g(x)$, respectively yields
\begin{equation} \label{e:um3} 
d(y,g(y))=d(x,g(y)) \quad \textrm{and} \quad d(x, g(x))=d(y,g(x)).
\end{equation}
Then inequalities~\eqref{e:um} and \eqref{e:um3} imply 
\begin{equation*}
d(x,g(x))\le d(x,g(y))=d(y,g(y))\le d(y, g(x))=d(x, g(x)),
\end{equation*}
so
\begin{equation} \label{e:um4}
d(x,g(x))=d(x,g(y))=d(y,g(y))=d(y, g(x)).
\end{equation}
Let $r=\dist(x,A)=d(x,g(x))$, then \eqref{e:um4} yields $r=d(y,g(y))=\dist(y,A)$. Then \eqref{e:um2} and \eqref{e:um4} imply 
$d(x,y)<d(g(x),g(y))\leq d(x,g(y))=r$. As $d(x,y)<r$, Fact~\ref{f:u} yields $S(x,r)=S(y,r)$. Hence by the definition of $g$ we have $g(x)=g(y)$, so $d(x,y)<d(g(x),g(y))=0$, which is a contradiction. The proof is complete.
\end{proof}

\begin{proof}[Proof of Lemma~\ref{l:extension}]
Since $Y$ is complete one can easily extend $f$ to the closure of $A$ as a Lipschitz function with the same Lipschitz constant, hence we can assume that $A$ is closed and so also compact. Let $g\colon X\to A$ be the Lipschitz-$1$ retract provided by Lemma~\ref{l:lipretract}. Then $f\circ g\colon X\to Y$ is clearly a Lipschitz extension of $f$ with the same Lipschitz constant.
\end{proof}

\section{Dimensions - proof of Corollary~\ref{c:dim}}\label{s:dim}

In this section we prove the following stronger form of Corollary~\ref{c:dim}.
Note that the technical assumptions about the domain $\iF$ of $D$ (see the first sentence and (v)) hold for any of the reasonable domains like all subsets, all Borel sets, all closed sets or all compact sets, so Corollary~\ref{c:dim} is indeed a special case of Theorem~\ref{t:dim}. 
Recall that a set is called $F_\sigma$ if it can be written as a countable union of closed sets.

\begin{theorem}\label{t:dim}
For a fixed positive integer $n$ let $\iF$ be any family of subsets of $\R^n$ that contains all compact subsets of $\R^n$.
Suppose that a function $D\colon \iF\to[0,n]$ has the following properties.
\begin{enumerate}[(i)] 
\item Lipschitz functions cannot increase $D$ for compact sets, that is, for any compact $K\su\R^n$ and Lipschitz function $f\colon K\to\R^n$ we have $D(f(K))\le D(K)$, 
\item  $D$ is monotone, that is, $A, B\in\iF$, $A\subset B$ implies that $D(A)\le D(B)$, and 
\item if $S\su\R^n$ is a homogeneous self-similar set with the SSC then $D(S)=\hdim S$.
\end{enumerate}
Then the following statements hold.
\begin{enumerate}[a)]
\item For every analytic $A\in\iF$ (in particular for any Borel $A\in\iF$) we have 
\begin{equation*} 
\hdim A\le D(A).
\end{equation*}
\item \label{b}
For any bounded $A\in\iF$ we have
\begin{equation*} D(A)\le \ubdim A.
\end{equation*}
\item 
If we also require that 

\begin{enumerate}[(i)]
\setcounter{enumii}{3}
\item \label{iv}
$D$ is $\sigma$-stable for compact sets, that is, 
$D(\cup_{i=1}^{\infty} K_i)=\sup_{i} D(K_i)$ 
whenever $\cup_{i=1}^\infty K_i\in\iF$ and every $K_i$ is compact, and
\item 
 $\iF$ contains all $F_\sigma$ sets or $A\cap F\in\iF$ for any $A\in\iF$ and closed $F\su\R^n$,
\end{enumerate}
then for every $A\in\iF$ we have
\begin{equation*}
D(A)\le \pdim A.
\end{equation*}
\end{enumerate}
\end{theorem}

\begin{proof}
a) 
Suppose that for some analytic $A\in\iF$ we have $\hdim A>D(A)$.
Then by Howroyd's theorem \cite{Ho} there exists a compact set $K\su A$ with $\hdim K > D(A)$. 
By the monotonicity of $D$ we have $D(K)\le D(A)$, so we obtain that $D(K)<\hdim K$.
Let $S\su\R^n$ be a homogeneous self-similar set with the SSC such that 
\begin{equation*} 
D(K)<\ubdim S <\hdim K.
\end{equation*}
By Theorem~\ref{t:dimimplieslip} there exists a compact set $K'\su K$ and a Lipschitz onto map $f\colon K'\to S$.
Then, using our assumptions we obtain
\begin{equation*} 
D(S)=D(f(K'))\le D(K')\le D(K)<\ubdim S = D(S),
\end{equation*}
which is a contradiction.

b) Suppose that for some bounded $A\in\iF$ we have $\ubdim A<D(A)$.
Let $K$ be the closure of $A$.
Then $K$ is compact and by the monotonicity of $D$ we get $\ubdim K=\ubdim A<D(A)\le D(K)$.
Let $S\su\R^n$ be a homogeneous self-similar set with the SSC such that 
\begin{equation*} 
\ubdim K <\hdim S < D(K).
\end{equation*}
By Corollary~\ref{c:dimwithextension} there exists a Lipschitz onto map $f\colon S\to K$, and our assumptions imply
\begin{equation*} 
D(S)=\hdim S < D(K)=D (f(S)) \le D(S),
\end{equation*}
which is a contradiction. 

c) Let $A\in\iF$.
Using \eqref{e:pdim} and the fact that taking closure does not change the upper box dimension we get that
\begin{equation*}
    \pdim A=\inf\left\{\sup_i \ubdim \overline{A_i}\colon A=\cup_{i=1}^\infty A_i,\, A_i \textrm{ is bounded}\right\}.    
\end{equation*}
So let $A=\cup_{i=1}^\infty A_i$ such that each $A_i$ is bounded.
Since each $\overline{A_i}$ is compact, part 
\eqref{b} implies
$\sup_i\ubdim\overline{A_i}\ge \sup_i D(\overline{A_i})$.
Therefore, to complete the proof it is enough to show that $\sup_i D(\overline{A_i})\ge D(A)$.

If $\iF$ contains all $F_\sigma$ sets then we can apply (iv) and (ii) to obtain
\begin{equation*}
\sup_i D(\overline{A_i})=D(\cup_{i=1}^\infty \overline{A_i})\ge D(A).
\end{equation*}
If $A\cap F\in\iF$ for any $A\in\iF$ and closed $F\su\R^n$ then we apply first (ii) then (iv) for the partition $A=\cup_{i=1}^\infty (A\cap\overline{A_i})$ to obtain
\begin{equation*}
\sup_i D(\overline{A_i})\ge 
\sup_i D(A\cap\overline{A_i})
=
D(A).
\end{equation*}
We considered the alternative conditions of (v), which completes the proof.
\end{proof}

\begin{remarks} \label{r:dimensionconditions}
(1) As one can see from the proof, we used only three facts about the family $\HH$ of homogeneous self-similar sets with the SSC:
\begin{enumerate}[(A)]
    \item \label{compactH} 
    every $S\in\HH$ is compact,
    \item \label{equaldim} 
    $\hdim S=\ubdim S$ for any $S\in\HH$ and
    \item \label{dense} 
    for any $0\le a<b\le n$ there exists a set $S\in\HH$ such that $a<\hdim S<b$.
\end{enumerate}
Thus in condition (iii) of Theorem~\ref{t:dim} the class of self-similar sets with the SSC can be replaced by any other family $\HH$ of compact subsets of $\R^n$ with the above properties.
For example, we could also use sets of the form $A \times \ldots \times A$, where $A$ is a middle-$c$ Cantor set for some rational $c$.

(2) For a fixed $\iF$ let
\begin{equation*}
D_1(A)=\sup\{\hdim K \colon K\su A \text{ compact}\}
\quad (A\in\iF).
\end{equation*}
Since part (a) of Theorem~\ref{t:dim} can be applied to any compact subset of $\R^n$, using also the monotonicity condition (ii), we obtain that (a) can be replaced by the following, more general statement:
\begin{itemize}
    \item[a')] For any $A\in\iF$ we have $D_1(A)\le D(A)$.
\end{itemize}

Since Hausdorff dimension clearly satisfies (i), (ii) and (iii), by (a) of Theorem~\ref{t:dim} it is the
smallest function on $\iF$ that satisfies (i), (ii) and (iii), provided $\iF$ contains only analytic sets.
Note that $D_1$ also satisfies conditions (i), (ii) and (iii), so on a general $\iF$ this is the smallest function that satisfy (i), (ii) and (iii).

Recall that a set $B\su\R^n$ is called a Bernstein set if neither $B$, nor $\R^n\sm B$ contains any uncountable compact subset. It is well known that such sets exist, see e.g.~\cite{Kec}. 
Then clearly $D_1(B)=D_1(\R^n\sm B)=0$ but 
$B$ or $\R^n\sm B$ has Hausdorff dimension $n$.
This shows that the condition that $A$ is analytic cannot be removed from Theorem~\ref{t:dim} (a). 

(3) The following example shows that in Theorem~\ref{t:dim} one cannot remove assumption (v) about the domain $\iF$ of $D$, not even if we require $\sigma$-stability for all sets of $\iF$, that is, if we replace (iv) by the following stronger assumption: 
\begin{itemize}
    \item [(iv')] $D(\cup_{i=1}^\infty A_i)=\sup_i D(A_i)$ whenever $\cup_{i=1}^\infty A_i\in\iF$ and $A_i\in\iF$ for all $i$.
\end{itemize}

Since there are only continuum many $F_\sigma$ sets and every set of cardinality continuum has more than continuum many subsets, there exists an unbounded non-$F_\sigma$ set $E\su\R^n$ such that $\pdim E<n$.
Let $\iF$ consist of $E$ and all compact subsets of $\R^n$ and let $D_2$ be the packing dimension for the compact sets and let $D_2(E)=n$. 
This $\iF$ and $D_2$ have all the required properties but (v), still $D_2(A)\le\pdim(A)$ does not hold for $A=E$. 

On the other hand, we claim that  
if we replace (iv) by the even stronger assumption that
\begin{itemize}
\item [(iv'')]
$D(A)\le\sup_{i} D(K_i)$
whenever $A\su\cup_{i=1}^\infty K_i$, $A\in\iF$ and every $K_i$ is compact 
\end{itemize}
then assumption (v) about $\iF$ can be dropped.
Indeed, in this case in the proof of (c)
the needed inequality $\sup_i D(\overline{A_i})\ge D(A)$ directly follows from (iv'').
\end{remarks}

\section{Lipschitz functions between self-similar sets}
\label{s:ssc}

In this section we prove Theorem~\ref{t:ssc}. 
We need some preparation. 
First we prove the following algebraic lemma,
which might be known but for completeness we present a proof.

\begin{lemma}\label{l:algebra}
Let $q>1$ be an integer, which is not of the form $n^k$ for any integers $n,k>1$ and let $r_1,\ldots,r_m$ be rational numbers. Then 
\begin{equation} \label{eq:ri}
 q^{r_1}+\ldots+q^{r_m}=1    
\end{equation}
implies that every $r_i$ is an integer.
\end{lemma}

\begin{proof}
Assume to the contrary that \eqref{eq:ri} holds and not every $r_i$ is an integer. Let $n$ be the smallest 
common divisor of $r_1,\ldots,r_n$.
Then each $r_i$ can be written as $r_i=k_i+\frac{p_i}{n}$, where $k_i\in\Z$ and $p_i\in\{0,1,\ldots,n-1\}$.
Define the algebraic number $z=q^{1/n}$ and the rational polynomial
\begin{equation*}
 R(x)=-1+\sum_{i=1}^m q^{k_i} x^{p_i}.
\end{equation*}
We obtain that $R(z)=0$ by \eqref{eq:ri}, and clearly $\deg(R)<n$. 
Since not every $r_i$ is an integer, $\max_i p_i>0$, so $\deg(R)>0$.
Hence the minimal polynomial of $z$ over $\QQ$ has degree less than $n$.

Finally, we prove that $P(x)=x^n-q$ is the minimal polynomial of $z$ over $\QQ$, which will be a contradiction. Clearly $P(z)=0$. By \cite[Theorem~9.1, Chapter~6]{L} the polynomial $Q(x)=x^u-v$ for positive integers $u,v$ is irreducible over $\QQ$ if and only if for every prime $p|u$ we have $v^{1/p}\notin \QQ$, which is equivalent to $v^{1/p}\notin \NN^+$. As $q>0$ is not a perfect power, this implies that $P$ is irreducible over $\QQ$. Therefore, $P$ is indeed the minimal polynomial of $z$, and the proof is complete.
\end{proof}

\begin{notation}\label{n:F}
For metric spaces (or subsets of metric spaces) $A$ and $B$
we denote by $F(A,B)$ the minimal number of Lipschitz-$1$ images of $A$ that can cover $B$, that is,
\begin{equation*}
F(A,B)=\min\left\{k\colon \exists f_1,\ldots,f_k\colon A\to B \text{ Lipschitz-$1$ s.t.}~\cup_{i=1}^k f_i(A)=B\right\}.
\end{equation*}

For a metric space $A$ and an $r>0$ we denote by $rA$ the
scaled copy of $A$ in which every distance is the $r$-multiple
of the corresponding distance in $A$.
\end{notation}

\begin{fff}\label{f:scalingF}
For any metric spaces $A$ and $B$ and $r>0$ we have
$F(rA,rB)=F(A,B)$.
\end{fff}

\begin{lemma}\label{l:unionofNrA}
    Let $A$ and $B$ be metric spaces and $r>0$. 
    Suppose that $A=\cup_{i=1}^N A_i$ such that every $A_i$
    is isometric to $rA$. Then
 \begin{equation*} 
F(A,B)\ge \ceil*{\frac{F(rA,B)}{N}}.
\end{equation*}
Furthermore, if $\dist(A_i,A_j)\ge\diam B$ for every $i\neq j$,
then 
\begin{equation*} 
F(A,B)=\ceil*{\frac{F(rA,B)}{N}}.
\end{equation*}
\end{lemma}
\begin{proof}
    Let $m=F(rA,B)$. If $B$ can be covered by less than $m/N$
    Lipschitz-$1$ images of $A$ then we get a cover of $B$ by 
    less than $m$ Lipschitz-$1$ images of $rA$, which would
    contradict $m=F(rA,B)$. Therefore, $F(A,B)\ge \ceil*{\frac{F(rA,B)}{N}}$.

    If for every $i\neq j$ we have $\dist(A_i,A_j)\ge\diam B$, 
    then from any $N$ Lipschitz-$1$ onto maps $f_i\colon rA\to B$  we can get a Lipschitz-$1$ onto map $f \colon A\to B$. Thus $F(A,B)\le \ceil{m/N}=\ceil*{\frac{F(rA,B)}{N}}$.
\end{proof}

\begin{lemma}\label{l:unionofnB}
Let $A$ and $B$ be metric spaces. Suppose that $B=\cup_{j=1}^n B_j$ and for every $i\neq j$ we have $\dist(B_i,B_j)>\diam A$. Then 
\begin{equation*}
F(A,B)=\sum_{j=1}^n F(A,B_j).    
\end{equation*}
\end{lemma}

\begin{proof}
Since $\dist(B_i,B_j)>\diam A$, no Lipschitz-$1$ image of $A$ can intersect more than one $B_i$. This implies the claim.
\end{proof}

Recall that a real valued function $f$ defined on a metric space $X$ is called \emph{lower semicontinuous}
if whenever $x, x_1, x_2,\ldots\in X$, $f(x_n)\le y$ for every $n$ and $x_n\to x$ then $f(x)\le y$.
It is well known that such a function has a minimum on any compact set.


\begin{lemma}\label{l:lsc}
For any compact metric spaces $A$ and $B$, the function $f(x)=F(xA, B)$ is lower semicontinuous on $(0,\infty)$.
\end{lemma}

\begin{proof} 
By definition, it is enough to prove that if $z_n\to z$ and $F(z_nA,B)\leq k$ for all $n$, then $F(zA,B)\leq k$. Let $f_{i,n}\colon z_nA\to B$ be Lipschitz-$1$ functions such that $\cup_{i=1}^{k} f_{i,n}(A)=B$ for all $n$. 
For each $n$ let $h_n\colon zA\to z_nA$ be the natural bijection between $zA$ and $z_nA$.
As the sequences $\{f_{i,n}\circ h_n\}_{n\geq 1}$ are equicontiuous for all $1\leq i\leq k$, using the Arzel\`a-Ascoli theorem $k$-times and passing to a subsequence we may assume that there are functions $f_i\colon zA\to B$ such that $f_{i,n}\circ h_n\to f_i$ uniformly for all $1\leq i\leq k$. Then clearly every $f_i$ is a Lipschitz-$1$ function. Finally, it is enough to see that $\cup_{i=1}^k f_i(zA)=B$, which easily follows from the uniform convergence and that $\cup_{i=1}^k (f_{i,n}\circ h_n)(zA)=B$ for all $n$. 
\end{proof}




Now we are ready to prove Theorem~\ref{t:ssc}.
We prove it in the following stronger form.

\begin{theorem}\label{t:strongssc}
Let $A$ be a homogeneous self-similar set with the strong separation condition, that is, $A$ is the disjoint union of $q$ similar $r$-scaled copies of $A$ for some $r>0$ real and $q>1$ integer. Let $k$ be the maximal positive integer such that $\root k \of q$ is integer.
Let $B$ be any self-similar set with the SSC such that $\hdim A=\hdim B$. Then the following statements are equivalent:
\begin{enumerate}[(i)]
\item \label{i1} $A$ can be mapped onto $B$ by a Lipschitz map;
\item \label{i2} All similarity ratios of $B$ are positive integer powers of 
$\root k \of r$; 
\item \label{i3} $A$ and $B$ are bilipschitz equivalent.
\end{enumerate}
\end{theorem}

\begin{proof}
If $q=n^k$ for some integers $n,k>1$ then it is easy to see that
$A$ is bilipschitz equivalent to a homogeneous self-similar set 
$A'=\sqcup_{i=1}^n f_i(A')$ with the
SSC such that each $f_i$ has similarity ratio $\root k \of r$.
Therefore we can suppose that $k=1$.

Let $s=\hdim A=\hdim B$.
Then $qr^s=1$.
Let the similarity ratios of $B$ be $\beta_1,\ldots,\beta_m$ and so
    $\beta_1^s+\ldots+\beta_m^s = 1$.
For every $i$ let $\al_i=-\log_q \beta_i^s>0$. Hence $\beta_i^s=q^{-\al_i}=(r^{\al_i})^s$, so $\beta_i=r^{\al_i}$ and 
\begin{equation}\label{e:betasum}
   q^{-\al_1}+\ldots+q^{-\al_m}=1.
\end{equation}

First we prove that \eqref{i2} implies \eqref{i3}.
Since $k=1$, \eqref{i2} implies that $\al_1,\ldots,\al_m$ are positive integers.
We claim that in this case \eqref{e:betasum} implies that $A$ can be also written as a self-similar set using similarity ratios $r^{\al_1},\ldots,r^{\al_m}$. 
Indeed, we can clearly suppose that $\al_1\le\ldots\le\al_m$.
Let $A_1\su A$ be an $r^{\al_1}$-scaled copy of $A$,
let $A_2\su A\sm A_1$ be an $r^{\al_2}$-scaled copy of $A$, and so on. 
One can easily check that  $\al_1\le\ldots\le\al_m$ and \eqref{e:betasum} imply that we never get stuck and finally we obtain a partition $A=A_1\cup\ldots\cup A_m$ such that
each $A_i$ is an $r^{\al_i}$-scaled copy of $A$, which completes the proof of the claim. 
Since $A$ and $B$ are self-similar sets with the SSC and they can be obtained using the same similarity ratios,
Lemma~\ref{sscbilipultra} implies that $A$ and $B$ are bilipschitz equivalent.

Since \eqref{i3} clearly implies \eqref{i1}, it remains to prove that \eqref{i1} implies \eqref{i2}.
So we suppose that $A$ can be mapped onto $B$ by a Lipschitz map 
and we need to show that 
every $\al_i$ is a positive integer.
Using \eqref{e:betasum} and that $q$ is not of the form $q=n^k$ for any integers $n,k>1$, 
by Lemma~\ref{l:algebra} it is enough to prove that every $\al_i$ is rational.
 
Let $g_1,\ldots,g_m$ denote the similarity maps of ratios $\be_1,\ldots,\be_m$,
respectively, such that $B=\sqcup_{j=1}^m g_j(B)$.
We can clearly suppose that 
\begin{equation}\label{e:diameterandgap}
    \min_{i\neq j}\dist(g_i(B),g_j(B))> \frac{\diam A}{r}. 
\end{equation}
Let
\begin{equation*}
z(t)=F(r^{-\{t\}} A, B) q^{\{t\}},   
\end{equation*}
where $\{\cdot \}$ denotes fractional part.
Then $z(t)$ is clearly a $1$-periodic function.

Now we claim that for any $t\in\R$ we have
\begin{equation}\label{e:z}
    z(t)\ge \sum_{j=1}^m q^{-\al_j} z(t+\al_j).
\end{equation}
Indeed, noting that 
$r^{\floor{t}}A$ is the union of 
$N=q^{\floor{t+\al_j}-\floor{t}}$ copies of
$r^{\floor{t+\al_j}} A$, 
we can apply Lemma~\ref{l:unionofNrA} and Fact~\ref{f:scalingF}  to get
\begin{align}\label{e:ize}
\begin{split}
F(r^{\floor{t}}A, r^{t}g_j(B)) q^{\{t\}} &\ge {q^{\floor{t}-\floor{t+\al_j}} F(r^{\floor{t+\al_j}}A, r^{t}g_j(B)}) q^{\{t\}} \\
&= F(r^{-\{t+\al_j\}}A, B) q^{\{t+\al_j\}-\al_j} \\
&=q^{-\al_j} z(t+\al_j) . 
\end{split}
\end{align}
By \eqref{e:diameterandgap} we can apply Lemma~\ref{l:unionofnB} to $r^{\floor{t}}A$
and $r^{t}B=\cup_{j=1}^m r^{t}g_j(B)$.
Using Fact~\ref{f:scalingF}, Lemma~\ref{l:unionofnB} and \eqref{e:ize}
we obtain
\begin{align*}
\begin{split}
z(t)&=F(r^{-\{t\}} A, B) q^{\{t\}}= F(r^{\floor{t}}A, r^{t}B) q^{\{t\}} \\
&=\sum_{j=1}^m F(r^{\floor{t}}A, r^{t}g_j(B))q^{\{t\}} 
\ge \sum_{j=1}^m q^{-\al_j} z(t+\al_j),  
\end{split}
\end{align*}
which completes the proof of \eqref{e:z}.

Choose $\delta\in(0,1)$ such that 
\begin{equation}\label{e:delta}
(\forall j=1,\ldots,m)~ \al_j\not\in\Z
\Longrightarrow \dist(\al_j,\Z)>\delta.
\end{equation}
By Lemma~\ref{l:lsc}, $z$ is lower semicontinuous
on $[0,1)$. This implies that $z$ has a minimum $w$ on $[0,1-\delta]$ obtained at a point $u\in [0,1-\delta]$.

First we consider the case when $w$ is
the minimum of $z$ on $\R$. Then \eqref{e:z}, the minimality of $w$, and \eqref{e:betasum} imply that
\begin{equation*}
w=z(u)\ge \sum_{j=1}^m q^{-\al_j} z(u+\al_j)\ge \sum_{j=1}^m q^{-\al_j} w = w,    
\end{equation*}
which implies that $z(u+\al_j)=w$ for every $j$.
If not every $\al_j$ is rational then, since $z$ is $1$-periodic, we obtain that
$z$ is constant on a dense set. 
But this is impossible, since $z(t)$ is the product of a nonzero integer valued
function and $q^{\{t\}}$.

It remains to consider the case when $w$ is not the minimum of $z$ on $\R$. 
Since $w$ is the minimum of the $1$-periodic $z$ on $[0,1-\delta]$, this implies that there exists a
$v\in(1-\delta,1)$ such that $z(v)<w$.
Then \eqref{e:delta} implies that for any non-integer $\al_j$ we have $\{v+\al_j\}\in [0,1-\delta]$ and so
$z(v+\al_j)\ge w>z(v)$.
Since for integer $\al_j$ clearly $z(v+\al_j)=z(v)$,
by \eqref{e:betasum} we get that if not every $\al_j$
is integer then 
\begin{equation*}
z(v)<\sum_{j=1}^m q^{-\al_j} z(v+\al_j),
\end{equation*}
which contradicts \eqref{e:z}. The proof is complete.
\end{proof}

\section{Open questions}\label{s:questions}

In Theorem~\ref{t:strongssc} the assumption that $A$ is homogeneous is essential since
for general self-similar sets with the SSC (ii) and (iii) are not equivalent, the characterization by Xi \cite{X} of bilipschitz equivalent self-similar sets with the SSC is much more complicated. 
But (i) and (iii) might be equivalent without assuming homogeneity; that is, we do not know if Theorem~\ref{t:ssc} holds for any self-similar sets with the SSC.

Another possible and natural way to improve Theorem~\ref{t:ssc} is to assume only that $A$ can be mapped onto a subset of $B$ of positive measure. We do not know if this stronger result holds. If the answer to both of the above questions is positive then one might ask if both improvements can be made at the same time.

\subsection*{Acknowledgments} Our initial motivation came from Zolt\'an Balogh's inspiring talk on generalizations of the analyst's traveling salesman problem. We are also indebted to Amlan Banaji, Udayan Darji, M\'arton Elekes, Jonathan M.~Fraser, Korn\'elia H\'era, Mikl\'os Laczkovich, Alex Rutar and Sascha Troscheit for some illuminating conversations.
The authors are also grateful to the anonymous referee for the careful reading of the manuscript and for the constructive suggestions that improved the presentation of the paper.

\end{document}